\documentclass[12pt]{amsart}

\usepackage{amssymb}
\usepackage{latexsym, amssymb}

\usepackage{dsfont}
\usepackage{amsmath}

\makeatletter
\renewenvironment{description}
               {\list{}{\labelwidth\z@ \itemindent-\leftmargin
                        }}
               {\endlist}

\makeatother

\makeatletter
\def\eqnarray{\stepcounter{equation}\let\@currentlabel=\theequation
\global\@eqnswtrue
\tabskip\@centering\let\\=\@eqncr
$$\halign to \displaywidth\bgroup\hfil\global\@eqcnt\z@
  $\displaystyle\tabskip\z@{##}$&\global\@eqcnt\@ne
  \hfil$\displaystyle{{}##{}}$\hfil
  &\global\@eqcnt\tw@ $\displaystyle{##}$\hfil
  \tabskip\@centering&\llap{##}\tabskip\z@\cr}

\def\endeqnarray{\@@eqncr\egroup
      \global\advance\c@equation\m@ne$$\global\@ignoretrue}

\def\@yeqncr{\@ifnextchar [{\@xeqncr}{\@xeqncr[5pt]}}
\makeatother

\newcommand{\vertspace}{\vskip10.0pt plus 4.0pt minus 6.0pt}

\newcounter{proofstep}
\newcommand{\nextstep}{\refstepcounter{proofstep}\vertspace \par 
          \noindent{\bf Step \theproofstep} \hspace{5pt}}
\newcommand{\firststep}{\setcounter{proofstep}{0}\nextstep}

\newtheorem{theorem}{Theorem}[section]

\newtheorem{lemma}[theorem]{Lemma}
\theoremstyle{definition}

\newtheorem{examples}[theorem]{Examples}
\newtheorem{remark}[theorem]{Remark}

\newcommand{\RRe}{\mathop{\rm Re}}

\newcommand{\R}{\mathop{\rm Re}}

\newcommand{\eps}{\varepsilon}
\newcommand{\la}{\lambda}

\newcommand{\Om}{\Omega}

\newcommand{\RR}{\mathbb{R}}
\newcommand{\Ri}{\mathbb{R}}
\newcommand{\CC}{\mathbb{C}}
\newcommand{\NN}{\mathbb{N}}
\newcommand{\Ni}{\mathbb{N}}
\newcommand{\ZZ}{\mathbb{Z}}
\newcommand{\fra}{\mathfrak{a}}

\newcommand{\supp}{\mathop{\rm supp}}
\newcommand{\one}{\mathds{1}}

\begin{document}
\bibliographystyle{tom}

\thispagestyle{empty}

\vspace*{1cm}
\begin{center}
{\Large\bf Partial spectral multipliers  and partial Riesz transforms for degenerate operators} \\[5mm]

\large A.F.M. ter Elst$^1$ and E.M.  Ouhabaz$^2$

\end{center}

\vspace{5mm}

\begin{center}
{\bf Abstract}
\end{center}

\begin{list}{}
\item

We consider degenerate differential operators 
$A = \displaystyle{\sum_{k,j=1}^d \partial_k (a_{kj} \partial_j)}$ on $L^2(\RR^d)$ 
with real symmetric bounded measurable coefficients. 
Given a function  $\chi \in C_b^\infty(\RR^d)$ (respectively,  $\Om$  a bounded Lipschitz domain) and suppose that 
$(a_{kj}) \ge \mu > 0$ a.e.\  on $ \supp \chi$  (resp.,  a.e.\  on  $\Om$).  
We prove a  spectral multiplier type result:  
if   $F\colon  [0, \infty) \to \CC$ is such that  
$\sup_{t > 0} \| \varphi(.) F(t .)  \|_{C^s} < \infty$ for some non-trivial 
function $\varphi \in C_c^\infty(0,\infty)$ and  some $s > d/2$ 
then $M_\chi F(I+A) M_\chi$ is weak type $(1,1)$ (resp.\  $P_\Om  F(I+A) P_\Om$ is weak type $(1,1)$).  We also prove boundedness on $L^p$ for all $p \in (1,2]$ of  the partial Riesz transforms $M_\chi \nabla (I + A)^{-1/2}M_ \chi$. The proofs  are  based on a criterion  for   a singular integral operator to be weak type $(1,1)$. 
\end{list}

\noindent

\vspace{5mm}
\noindent
January 2012

\vspace{5mm}

\vspace{5mm}

\vspace{5mm}
\noindent

\noindent
{\bf Home institutions:}    \\[3mm]
\begin{tabular}{@{}cl@{\hspace{10mm}}cl}
1. & Department of Mathematics  & 
  2. & Univ.  Bordeaux,  \\
& University of Auckland   & 
  &IMB, CNRS  UMR 5251,  \\
& Private bag 92019 & 
  &351, Cours de la Lib\'eration  \\
& Auckland & 
  &  33405 Talence \\
& New Zealand  & 
  & France  \\[8mm]
\end{tabular}

\section{Introduction} 
Let $A$ be a non-negative self-adjoint uniformly elliptic operator in divergence  form.
 More precisely, let $a_{kj} = a_{jk} \colon \Ri^d \to \Ri$ be a 
bounded measurable function
for all $j,k \in \{ 1,\ldots,d \} $ and assume  that there exists a $\mu > 0$
such that 
\begin{equation}\label{I1}
\sum_{k,j=1}^d a_{kj}(x) \xi_k \xi_j 
\ge \mu |\xi |^2  \mbox{ for all } \xi = (\xi_1,\ldots,\xi_d) \in \RR^d \mbox{ and } x \in \Ri^d
\end{equation}
The operator $A = - \displaystyle\sum_{k,j=1}^d \partial_k(a_{kj} \partial_j)$, 
defined by quadratic form techniques, is self-adjoint on $L^2(\RR^d)$.
It is a standard fact that 
$-A$ is the generator of a strongly continuous semigroup $(e^{-tA})_{t >0}$ on $L^2(\RR^d)$.
The well known Aronson  estimates assert that $e^{-tA}$ is given an integral kernel 
$p_t$ (called the heat kernel of $A$) which satisfies the Gaussian upper bound:
\begin{equation}\label{I2}
| p_t(x,y) | \le Ct^{-d/2} e^{-c \frac{|x-y|^2}{t}} \mbox{ for all } t > 0 \mbox{ and } x, y  \in \RR^d.
\end{equation}
Here $C$ and $c$ are positive constants.  

In recent years,  harmonic analysis of operator of type $A$ has attracted a 
lot of attention and substantial progress have been made in which upper  bounds
for the heat kernel play  a  fundamental  role.
We mention for example the theory of Hardy and BMO spaces associated with 
such operators (see  for example \cite{DY} and \cite{HLMMY}), spectral multipliers \cite{DOS} 
and Riesz transforms
(see \cite{DuM2}, \cite{Aus2}, \cite{Ouh5}, \cite{She} and the references therein).
 Concerning spectral multipliers, it is known that if $F\colon  [0, \infty) \to \CC$ is a 
bounded measurable function then the operator $F(A)$,
which is well defined on $L^2$ by spectral theory, extends to a bounded operator 
on $L^p$ for all $1 < p < \infty$ provided $F$ satisfies the condition
\begin{equation}\label{I3}
\sup_{t> 0} \| \varphi (.) F(t.)  \|_{C^s} < \infty
\end{equation} 
for some $s > d/2$ and some non-trivial auxiliary function $ \varphi \in C_c^\infty (0,\infty)$.
See Duong--Ouhabaz--Sikora \cite{DOS}, where a more general result is proved.
  Note that 
condition (\ref{I3}) is satisfied if $F$ has $[d/2]+1$ derivatives such that 
$$\sup_{\la > 0} \lambda^k | F^{(k)}(\la) | < \infty \mbox{ for all }  k \in \{ 0, 1,\ldots , [d/2] +1 \} .$$
As an example, one obtains  polynomial estimates on $L^p$  for imaginary powers  of type  
$\| A^{is} \|_{p \to p} \le C (1 + |s |)^{\beta_p}$ for all $\beta_p > d | \frac{1}{2} - \frac{1}{p}|$.
Taking $F(\la) := (1 - \frac{\lambda}{R})_+^\alpha$ one obtains  Bochner--Riesz summability for all $\alpha > d/2$. 

Concerning  Riesz transforms  ${\mathcal R}_k := \partial_k A^{-1/2}$, it is an obvious consequence of the  ellipticity assumption (\ref{I1})  that 
${\mathcal R}_k$ is  bounded on $L^2(\RR^d)$ for all $k \in \{ 1,\ldots ,d \} $.
 As for multiplier results, the Gaussian bound (\ref{I2}) combined with recent developments  on singular integral operators
allow to prove that ${\mathcal R}_k$ is bounded on $L^p(\RR^d)$ for all $p \in (1, 2)$ with the sole assumption (\ref{I1}) and bounded measurable coefficients 
(see Duong--McIntosh \cite{DuM2},  Auscher \cite{Aus2},  Ouhabaz \cite{Ouh5}).
Under weak  regularity assumption on the coefficients one obtains boundedness of 
${\mathcal R}_k$ on $L^p(\RR^d)$ for all $p \in (2,\infty)$ (cf.\  Auscher \cite{Aus2}, Shen \cite{She}). 

In the present paper we wish to study similar problems for degenerate operators.
 Instead of (\ref{I1}) we merely assume that 
\begin{equation}\label{I4}
\sum_{k,j=1}^d a_{kj}(x) \xi_k \xi_j \ge 0 \mbox{ for  all } \xi = (\xi_1,\ldots ,\xi_d) \in \RR^d
   \mbox{ and } x \in \Ri^d.
\end{equation}
In this case, we  define  the form
\begin{equation} \label{formcc}
\fra_0(u,v) = \sum_{k,j=1}^d \int_{\RR^d} a_{kj} \, (\partial_j u) \, (\partial_k v) 
\end{equation}
with form domain $D(\fra_0) = C_c^\infty(\RR^d)$.
If this form is closable, then  $A$ will be the self-adjoint operator associated with its closure.
If not, we take the regular part and consider $A$ as the operator associated with the closure of this 
regular part (see \cite{bSim5} and \cite{AE2}).  

Proving results like the previous ones for  these operators seems unattainable because 
 Gaussian (or Poisson) upper bounds  are not true in general.
Even $L^1$-$L^\infty$ estimates of $e^{-tA}$ are not valid in general.
What we will do is to restrict 
 the operators to parts where the matrix $(a_{kj})$ is elliptic.
 It is proved by ter Elst and Ouhabaz \cite{EO1} that if 
$\chi \in C_{\rm b}^\infty(\RR^d)$ and $\mu > 0$ are such that 
$(a_{kj}(x)) \ge \mu I $ for a.e.\  $x \in \supp \chi$, 
then $M_\chi e^{-tA} M_\chi$ has a H\"older continuous kernel 
$K_t$ which satisfies the Gaussian bound 
\begin{equation}\label{I5}
| K_t(x,y) | \le Ct^{-d/2} e^{-c \frac{|x-y|^2}{t}}  (1 + t)^{d/2} \mbox{ for all } t > 0 \mbox{ and } x, y  \in \RR^d.
\end{equation}
Here $M_\chi $  is the multiplication operator by $\chi$.
The same result holds  for $P_\Om e^{-tA} P_\Om$ if $\Om$ is a bounded Lipschitz domain such that 
$(a_{kj}(x)) \ge \mu I$  for a.e.\ $x \in \Om$ for some $\mu > 0$.
Here $P_\Om$ is the multiplication operator by the indicator function $\one_\Om$ of $\Om$. 

Note that in general one cannot get rid of the extra term $(1+ t)^{d/2}$ in the 
right hand side of (\ref{I5}).
For example, if 
$a_{kj} = \delta_{kj}$ on a smooth bounded domain $\Om$, then $A$ is the Neumann Laplacian on $L^2(\Om)$ and $0$ on $L^2(\RR^d\setminus \Om)$.
It is then clear 
 that the Gaussian bound is not valid without the additional term $(1+t)^{d/2}$.  
Because of  that additional term in (\ref{I5}),  we shall  consider in the sequel 
$ I + A$ instead of $A$ (of course,  one can take $\eps I + A$ for any  $\eps > 0$ to absorb 
the factor $(1+ t)^{d/2}$).

For spectral multipliers and Riesz transforms we will prove the following results.
Given $\chi \in C_{\rm b}^\infty(\RR^d)$  (resp.,  $\Om$ a bounded Lipschitz domain) 
such that $(a_{kj}(x)) \ge \mu I $  for a.e.\   $x \in \supp \chi$ (resp., for a.e.\  $x \in \Om$) 
for some constant $\mu > 0$.
The main theorems of this paper are the 
following.

\begin{theorem} \label{theoremA}
Let  $F\colon  [0, \infty) \to \CC$ be a bounded function such that 
$$\sup_{t> 0} \| \varphi (.) F(t.) \|_{C^s} < \infty$$
for some $s > d/2$ and some non-trivial function $\varphi \in C_c^\infty(0,\infty)$.
Then $M_\chi F(I+A) M_\chi$ {\rm (}resp., $P_\Om F(I+A) P_\Om${\rm )} is bounded 
on $L^p(\RR^d)$ for all $1 < p < \infty$. 
\end{theorem}

\begin{theorem} \label{theoremB}
 For all $k \in \{  1, \ldots , d \} $ the Riesz transform type operator 
$M_\chi \partial_k (I+ A)^{-1/2}M_\chi$ is  bounded on $L^p(\RR^d)$ for all
$1 < p \le 2$.
\end{theorem}

Now we discuss how we prove these results.
In the elliptic case, besides the Gaussian bound (\ref{I2}), the proof of boundedness of spectral multipliers or Riesz transforms
rely on a criterion proved by Duong and McIntosh \cite{DuM3} for singular integral operators to be weak type $(1,1)$. 
This criterion says that if $T$ is bounded on $L^2$ with a (singular) kernel~$K$ 
such that there exists a family of operators $(A_t)_{t > 0}$ given by integral kernels $a_t$ which satisfy Gaussian (or Poisson) bounds, $TA_t$ is also
given by a (singular) kernel $K_t$ and 
there are $C,\delta > 0$ such that 
\begin{equation}\label{I6}
\int_{|x-y| \ge \delta \sqrt{t}} | K(x,y) - K_t(x,y) | \, dx \le C 
\end{equation}
for all $t > 0$ and a.e.\ $y$,
then $T$ is weak type $(1,1)$.
In applications to spectral multipliers  of  elliptic operators we start with  $T = F(A)$ and  one takes  $A_t = e^{-tA}$.
 Therefore, $K_t$ is the kernel of  the operator $F(A) e^{-tA }$ which can be seen as  a regularization   of   $F(A)$.
 In the degenerate case and because of (\ref{I5}), it is tempting to choose 
 $A_t = M_\chi e^{-t(I+A)}M_\chi$.
Then, 
 $$ T A_t  =  M_\chi F(I+A) M_\chi M_\chi e^{-t(I+A)} M_\chi =  M_\chi F(I+A) M_\chi^2 e^{-t(I+A)} M_\chi.$$ 
Now, the presence of $M_\chi^2$ does not allow to regularize $F(I+A)$ by $e^{-t(I+A)}$.
The  simple fact that we do not have $F(I+A)$ next to $e^{-t(I+A)}$ in the expression of 
$TA_t$  destroys this strategy.
The same problem occurs for the Riesz transform $M_\chi \partial_k (I+A)^{-1/2} M_\chi$.
 To overcome this difficulty we prove a version of the 
Duong--McIntosh criterion that is  suitable for our purpose.
It  reads as follows (see Theorems \ref{th1} and \ref{th2} together
with Remark~\ref{rpartmult202} for precise and quantitive statements). 

\begin{theorem} \label{theoremC}
 Let  $T$ be a bounded linear operator on $L^2$ and   $(A_t)_{t > 0}$  
a family of linear operators which satisfy  $L^1$-$L^2$ off-diagonal estimates.
Suppose that  there exists a bounded linear operator $S $ on $L^2$  and 
$\delta,W > 0$ such that 
\begin{equation}
\int_{|x-y| \geq (1 + \delta) t} | (T  - SA_t) u(y) | \, dy 
\le W \| u \|_1
\label{etheoremC;1}
\end{equation}
for all $x \in \RR^d$, $t > 0$ and  $u \in L^1 \cap L^\infty$ supported  in the ball $B(x, t)$. Then 
$T$ is weak type $(1,1)$. 
\end{theorem}
 
Note that the estimate (\ref{etheoremC;1})
is satisfied if  $T$ and $SA_t$ are given 
by {\rm (}singular{\rm )}  kernels $K$ and  $K_t$ and there are $C,\delta > 0$ such that 
 $$\int_{|x-y| \ge \delta \sqrt{t}} | K(x,y) - K_t(x,y) | \, dx \le C$$
for all $t > 0$ and a.e.\ $y \in \RR^d$.

Theorem~\ref{theoremC} gives more freedom by choosing any appropriate operator $S$ and not necessarily $S = T$. 
Coming back to spectral multipliers for degenerate operators $A$, we had 
$T = M_\chi F(I+A) M_\chi $ and we choose  now $S = M_\chi F(I+A) $ and $A_t = e^{-t(I+A)} M_\chi$.
Then $TA_t = M_\chi F(A+I) e^{-t(I+A)} M_\chi$ for which we can prove the estimate in Theorem~\ref{theoremC}.
Similarly for the Riesz transforms where $T = M_\chi \partial_k (I+A)^{-1}M_\chi$, we take $S = M_\chi \partial_k (I+A)^{-1}$ which turns to be bounded on $L^2$ and 
$A_t = e^{-t(I+A)} M_\chi$.
We emphasize   that $A_t = e^{-t(I+A)} M_\chi$ satisfies $L^1$-$L^2$ 
off-diagonal estimates but it is not known whether it satisfies Gaussian upper bounds in 
general\footnote{ under the additional assumption that $a_{kj} \in 
W^{1,\infty} (\RR^d)$, we proved recently in \cite{EO2} that $e^{-t(I+A)} M_\chi$ has a kernel which satisfies  a Gaussian bound.}.
 We believe that our version of the Duong--McIntosh criterion can be used in other circumstances in which a products of several operators come into play.
Also, as in \cite{DuM3}, our version holds for operators on domains of spaces of homogeneous type.

\smallskip

\noindent {\bf Notation.}  
We fix some notation which we will use throughout this paper.
 If $(X, \rho, \mu)$ is  a metric measured space, $x \in X$, $r > 0$ and $j \in \Ni$,
then we denote by 
$B(x,r) :=  \{ y \in X : \rho(x,y) < r \} $ the open ball of $X$ with  centre $x$ and radius $r$, 
the annulus  $C_j(x,r) = B(x, 2^{j+1}r)\setminus B(x, 2^jr)$ if $j \ge 2$ and $C_1(x,r) = B(x, 4r)$.
Let $v(x,r) = \mu(B(x,r))$ be the volume of the ball $B(x,r)$.
Next,  $\| T \|_{p \to q}$ is the norm of $T$ as an operator from $L^p$ to $L^q$.
If $E$ is a measurable set, then $P_E$ denotes the multiplication operator
by the indicator function $\one_E$ of $E$.
If $s \in (0,\infty) \setminus \Ni$ then we denote by $C^s$ the space of all Lipschitz functions  
on $[0,\infty)$ of order $s$ 
(i.e., functions which are continuously differentiable up to  $[s]$ and the derivative of order 
$[s]$ is H\"older continuous of order  $s-[s]$).
By $W^{r,p}$ we denote the classical Sobolev spaces on $\RR^d$.

All our operators are linear operators.

We emphasize that we shall use $C, C', c,\ldots $ for all inessential constants.
A constant $C$ may differ from line to line, even within one line.

\section{Singular integral operators}

Let $(X, \mu, \rho)$ be a metric measured space.
 We shall assume that $0 < v(x, r) < \infty$ for all $x \in X$ and $r > 0$ 
and that $X$  is a space of homogeneous type. 
This means that it satisfies the following doubling condition
 \begin{equation}\label{doubl1}
 v(x,2r) \le C_{0} v(x, r) 
 \end{equation}
for some $C_0 > 0$, uniformly for all $x \in X$ and $r > 0$.
If  (\ref{doubl1}) is satisfied then  there exist
 positive constants $C_{1}$ and $d$ such that 
 \begin{equation}\label{doubl2}
 v(x,\la r) \le C _{1}\la ^d v(x, r) 
 \end{equation}
for all $x \in X$ and $r \geq 1$.
Let $\Om$ be an open subset of $X$.
It is endowed with $\rho$ and $\mu$ but $(\Om, \mu, \rho)$ is not necessarily a space of homogeneous type. 
Let $T$ be a bounded linear operator on $L^{p_0}(\Om) := L^{p_0}(\Om, \mu)$ for some $p_0 \in [1, \infty)$.
We say that $T$ is given by a kernel $K \colon  \Om \times \Om \to \CC$ 
if $K$ is measurable and 
\begin{equation}\label{kernel}
Tu(x) = \int_\Om K(x,y) \, u(y) \, d\mu(y)
\end{equation}
for all $u \in L^{p_0}(\Om)$ with bounded support and a.e.\  $x$ outside the support of $u$.
We also say that $K$ is the associated kernel of $T$. 
A classical problem in harmonic analysis is to find conditions on the 
kernel $K$ which allow to extend the operator $T$ from $L^{p_0}(\Om)$ to other $L^p(\Om)$-spaces. 
Several results are known in this direction.
We refer the reader to  \cite{Ste3}, \cite{DuM3}, \cite{BK1}, \cite{Aus2} and the references therein. 

The main result in \cite{DuM3} states  that if there exists a family of bounded operators 
$A_t$, with $t > 0$, which are given by  integral kernels $a_t$ satisfying a Gaussian or Poisson 
 estimate and if the  associated  kernel  of $T - TA_t$ does not oscillate too much in a certain  sense then $T$ is weak type $(1,1)$.
Here we prove by the same method that if there exists a bounded operator $S$ on 
$L^{q_0}(X)$ for some $q_0 \in (1, \infty)$  such that  the  associated  kernel of $T - SA_t$ 
does not oscillate too much  then $T$ is weak type $(1,1)$.
  As explained in the introduction, this new version gives more  freedom by choosing  any appropriate $S$ which  may not coincide with $T$.
  This extension turns out to be powerful to prove    spectral multiplier type results  as well as Riesz 
transforms  for degenerate  operators,
whereas it is not clear how to apply the  condition from \cite{DuM3}.  
Note also that, following ideas from \cite{BK1} and \cite {Aus2} we can weaken the assumption on the kernel of $A_t$. 
Instead of assuming a Gaussian or Poisson bound, we  merely assume  an  $L^1$-$L^{q_0}$ off-diagonal estimate (see (\ref{au2}) below).
This  difference is again illustrated in our application to degenerate operators. 
In addition it is possible to formulate the result in \cite{DuM3} without reference
to the kernels (see also the remark immediately after the next theorem).

We first state and prove the result in the case $\Om = X$. 

\begin{theorem}\label{th1}
Let  $T$ be  a non-zero bounded linear operator on $L^{p_0}(X)$ for some $p_0 \in (1, \infty)$.
 Suppose there exists a bounded linear  operator $S$ on $L^{q_0}(X)$ for some $q_0 \in (1, \infty)$, 
 a family of
 bounded linear operators $(A_t)_{t > 0}$ on $L^{q_0}(X)$ 
and a sequence $(g(j))_{j \in \Ni}$ in $\Ri$ 
such that 
\begin{equation}\label{au2}
 \left(\frac{1}{v(x,2^{j+1}t)}\int_{C_j(x,t)} | A_{t}f |^{q_0} \right)^{1/q_0}
\le g(j) \frac{1}{v(x,t)}\int_{B(x,t)} | f | 
\end{equation}
for  all  $x \in X$, $t > 0$, $j \in \Ni$ and  $f \in L^{q_0}(B(x,t))$,
and  $\sum_{j=1}^\infty 2^{jd}g(j) < \infty$. 
Finally, suppose there exist $\delta,W > 0$ such that  
\begin{equation}\label{DMop}
\int_{X \setminus B(x,(1+\delta)t )} | (T  - SA_t) u |  \le W \| u \|_1
\end{equation} 
for all $x \in X$, $t > 0$ and  $u \in L^1(X) \cap L^\infty(X)$ supported  in the ball $B(x, t)$.  
Then $T$ is a weak type $(1,1)$ operator with
\begin{equation}\label{L1} 
\| T \|_{L^1 \to L^{1,w}} 
\le C (1+\delta)^d \left( W + \| T \|_{p_0 \to p_0} 
   + \| S \|_{q_0 \to q_0}^{q_0}  \| T \|_{p_0 \to p_0}^{1-q_0} \right).
\end{equation}
Here $C$ is a constant depending only on the constants in \eqref{doubl2}.
In particular, $T$ extends to a bounded operator on $L^p(X)$ for all $p \in (1, p_0)$.
\end{theorem}

\begin{remark} \label{rpartmult202}
Let $p_0,q_0 \in (1,\infty)$, 
$T \in \mathcal{L}(L^{p_0}(X))$ and for all $t > 0$ let 
$S,A_t \in \mathcal{L}(L^{q_0}(X))$.
Suppose that 
$T$ and $S \, A_t$ have a kernel $K$ and $K_t$, respectively.
Let $\delta,W > 0$ and assume that 
\begin{equation}\label{DM}
\int_{\rho(x,y) \ge \delta t} | K(x,y) - K_t(x,y) | \, d\mu(x) \le W < \infty,
\end{equation} 
for all $t > 0$ and $y \in X$. 
Fix now $x \in X$, $t > 0$ and  $u \in L^1(X) \cap L^\infty(X)$ supported  in the ball $B(x, t)$.  
Then  
\begin{eqnarray*}
&&\hspace{-2cm}\int_{X \setminus B(x,(1+\delta)t )} | (T  - SA_t) u(y) | \, d\mu(y) \\
 &=& \int_{X \setminus B(x,(1+\delta)t )} | \int_{B(x,t)} \Big( K(y,z) -K_t(y,z) \Big)  u(z) \, d\mu(z) | \, d\mu(y)\\
&\le&  \int_{X }  \int_{ \rho(y,z) \ge \delta t} | K(y,z) -K_t(y,z) |  \, d\mu(y) \, |u(z)| \, d\mu(z)\\
&\le& W \| u \|_1.
\end{eqnarray*}
Thus, (\ref{DMop}) is satisfied. The  condition \eqref{DM} is the direct analogue of the 
condition in Duong--McIntosh \cite{DuM3}. 

We also observe that  one does not  need  kernels   for both operators $T$ and $TA_t$ but a kernel $H_t(x,y)$ for the difference 
$T- SA_t$.  We may then replace $K(x,y) -K_t(x,y)$ in  (\ref{DM}) by $H_t(x,y)$. On the other hand the local estimate (\ref{DMop}) which does not appeal 
to kernels may have   advantage of avoiding measurability questions with respect to $x$ and $y$ of the expected singular kernels. 
\end{remark}

\begin{proof} 
As mentioned before, the arguments are similar to the arguments  used in \cite{DuM3} and \cite{Aus2}.
 We give the details for 
 convenience.
Recall we denote by $C$ all inessential constants.
 
 We begin by the classical Calder\'on--Zygmund decomposition.
There exist $c,N > 0$ such that the following is valid.
Fix $f \in L^1(X) \cap L^\infty(X)$ and $\alpha > \frac{\| f \|_1}{\mu(X)}$. 
 There exist  $g,b_1,b_2,\ldots \in L^1(X) \cap L^\infty(X)$ such that 
\[
f = g + b = g + \sum_i b_i
\]
 and
\begin{itemize}
\item[(i)]  $| g(x) | \le c \alpha$ for a.e.\  $ x \in X$,
\item[(ii)] each $b_i$ is supported in a ball $B_i = B(x_i, r_i)$ 
and $\frac{\| b_i \|_1}{v(x_i,r_i)}  \le c \alpha $,
\item[(iii)] $\sum_i v(x_i, r_i) \le c \frac{\| f \|_1}{\alpha}$, and,
\item[(iv)] there exists a constant $N$ such that $\sum_i \one_{B_i}(x) \le N$ for a.e.\  $x \in X$. 
\end{itemize}
 See Section~III.2 in \cite{CW}.  
 
 We proceed in several steps.

\firststep
Using the boundedness of $T$ on $L^{p_0}$ we have
 \begin{eqnarray*}
 \mu(\{ x \in X :  | (Tg)(x) | > \alpha \} ) 
&\le&\frac{ \| T \|_{p_0 \to p_0}^{p_0}}{\alpha^{p_0}}  \| g \|_{p_0}^{p_0}\\
 &\le& C \alpha^{p_0 -1} \frac{ \| T \|_{p_0 \to p_0}^{p_0}}{\alpha^{p_0}}  \| g \|_1\\
 &\le& C  \frac{ \| T \|_{p_0 \to p_0}^{p_0}}{\alpha}   \|  g \|_1. 
 \end{eqnarray*}
 It follows from (ii) and (iii) above that $\| b \|_1 \le c \| f \|_1$ and hence 
$\| g \|_1 \le (1 + c) \| f \|_1$. 
 Therefore, 
 \begin{equation}\label{eq1}
 \mu(\{ x \in X : | (Tg)(x) | > \alpha \} ) \le C  \frac{ \| T \|_{p_0 \to p_0}^{p_0}}{\alpha}  \| f \|_1.
 \end{equation}

\nextstep
We shall prove that 
 \begin{equation}\label{eq2}
 \| \sum_i A_{r_i} b_i \|_{q_0} \le C  \alpha^{1- \frac{1}{q_0}} \| f \|_1^{1/{q_0}}.
 \end{equation}
 We follow similar arguments as in \cite{Aus2}.
Fix  $u \in L^{q_0'}$ with $\| u \|_{q_0'} = 1$, where $q_0'$ is the dual exponent of $q_0$.
Let $i,j \in \Ni$ and set  $C_{i,j} := C_j(x_i,r_i)$. 
Then 
 \begin{eqnarray*}
 \int_{C_{i,j}} | A_{r_i} b_i | \, | u | 
&\le& \Big(  \int_{C_{i,j}} | A_{r_i} b_i |^{q_0} \Big)^{1/{q_0}} 
    \Big( \int_{C_{i,j}} |u |^{q_0'} \Big)^{1/{q_0'}}\\
 &\le& g(j) \frac{v(x_i, 2^{j+1}r_i)^{1/{q_0}}}{v(x_i,r_i)} \Big( \int | b_i | \Big) 
     \Big( \int_{C_{i,j}} |u |^{q_0'} \Big)^{1/{q_0'}}\\
 &\le & c \alpha  g(j) v(x_i, 2^{j+1} r_i)  
    \Big(  \frac{1}{v(x_i, 2^{j+1}r_i)} \int_{C_{i,j}} |u |^{q_0'} \Big)^{1/{q_0'}}, 
 \end{eqnarray*}
 where we have used (\ref{au2}) and property (ii) in the Calder\'on--Zygmund decomposition. 
 Denote by ${\mathcal M}$ the Hardy--Littlewood maximal operator.
Then 
 $$  \frac{1}{v(x_i, 2^{j+1}r_i)} \int_{C_{i,j}} |u |^{q_0'} \le {\mathcal M}(| u |^{q_0'})(y)$$
for all $y \in B_i$.
Combining the previous inequalities and using the doubling condition (\ref{doubl2})
one estimates
$$ \int_{C_{i,j}} | A_{r_i} b_i | \, | u | 
\le C \alpha  2^{j d} g(j) v(x_i,r_i)  \left( {\mathcal M}(| u |^{q_0'})(y) \right)^{1/{q_0'}}.$$
Taking the integral over $y \in B_i$ gives 
$$  \int_{C_{i,j}} | A_{r_i} b_i | \, | u | 
\le C \alpha  2^{j d} g(j) \int_{B_i}  \left( {\mathcal M}(| u |^{q_0'})(y) \right)^{1/{q_0'}}d\mu(y).$$ 
 We  sum over $j $, $i$ and use  $\sum_j 2^{jd} g(j) < \infty$ 
together with property (iv) in the Calder\'on--Zygmund decomposition to  obtain
 \begin{eqnarray*}
 \int_X  |\sum_i A_{r_i} b_i | \, |u | 
&\le&  C \alpha  \int_X  \one_{\cup_i B_i}(y)   \left( {\mathcal M}(| u |^{q_0'})(y) \right)^{1/{q_0'}} d\mu(y)\\
& \le& C \alpha \|  \one_{\cup_i B_i} \|_{q_0}  \, \Big\|  
   \left( {\mathcal M}(| u |^{q_0'}) \right)^{1/{q_0'}} \Big\|_{q_0',w}\\
 &\le& C \alpha \Big(\sum_i v(x_i,r_i) \Big)^{1/{q_0}} \, \|  | u |^{q_0'} \|_1^{1/q_0'}.
 \end{eqnarray*}
 Note that we have used the fact that ${\mathcal M}$ is weak type $(1,1)$ to obtain the last inequality.
Using now (iii) of the Calder\'on--Zygmund decomposition and 
 $\| u \|_{q_0'} = 1$, we obtain (\ref{eq2}). 
 
 By assumption, $S$ is bounded on $L^{q_0}$.
Hence
 $$ \mu( \{ x \in X : | (S\sum_i A_{r_i} b_i) (x) | > \alpha \} ) 
\le \frac{1}{\alpha^{q_0}} \| S \|_{q_0 \to q_0}^{q_0} 
 \| \sum_i A_{r_i} b_i \|_{q_0}^{q_0}.$$
 Now we use (\ref{eq2})  to obtain
 \begin{equation}\label{eq3}
 \mu( \{ x \in X : | (S\sum_i A_{r_i} b_i) (x) | > \alpha \} ) 
\le \frac{C}{\alpha} \| S \|_{q_0 \to q_0}^{q_0} \| f \|_1.
 \end{equation}

\nextstep
Let $\delta$ be as in (\ref{DMop}) and for all $i \in \Ni$ 
set $Q_i := B(x_i, (1+ \delta) r_i)$, the ball of centre $x_i$ and radius $(1+ \delta)r_i$.
Then 
\begin{eqnarray*}
\lefteqn{
\mu(\{ x \in X : | \sum_i (T -SA_{r_i})b_i (x) | > \alpha \})
} \hspace{10mm} \\
 &\le& \sum_i \mu(Q_i) 
   +  \mu(\{ x \in X\setminus \bigcup_j Q_j : | \sum_i ((T -SA_{r_i})b_i) (x) | > \alpha \}) \\
 &\le& C (1+ \delta)^d  \sum_i v(x_i,r_i) 
   + \frac{1}{\alpha} \int_{X\setminus \bigcup_j Q_j} | \sum_i ((T - SA_{r_i})b_i) (x)| \, d\mu(x) \\
 &\le& \frac{C (1+ \delta)^d}{\alpha} \| f \|_1 
+  \frac{1}{\alpha} \sum_i \int_{ X\setminus Q_i}  | ((T - SA_{r_i})b_i) (x)| \, d\mu(x) \\
 &\le& \frac{C (1+ \delta)^d}{\alpha} \| f \|_1 + \frac{W }{\alpha} \sum_i \int |b_i(y)| \, d\mu(y) \\
& \le &  \frac{C (1+ \delta)^d (1+ W)}{\alpha} \| f \|_1.
 \end{eqnarray*}
 Note that the penultimate inequality follows from assumption (\ref{DMop}) and the last one from 
properties (ii) and (iii) in the Calder\'on--Zygmund decomposition.  Hence
  \begin{equation}\label{eq4}
   \mu(\{ x \in X : | \sum_i ((T -SA_{r_i})b_i) (x) | > \alpha \}) 
\le \frac{C (1+\delta)^d (1+ W)}{\alpha} \| f \|_1.
   \end{equation}

\nextstep
It follows from  (\ref{eq1}) that 
   \begin{eqnarray*}
\lefteqn{
\mu(\{ x \in X : | (Tf)(x) | > \alpha \}) 
} \hspace{10mm} \\*
     &\le&  \mu(\{ x \in X : | (Tg)(x) | > \frac{\alpha}{2} \}) +  \mu(\{ x \in X : | (Tb)(x) | > \frac{\alpha}{2} \})\\
    &\le& C  \frac{ \| T \|_{p_0 \to p_0}^{p_0}}{\alpha}  \| f \|_1 + \mu(\{ x \in X : | (Tb)(x) | > \frac{\alpha}{2} \}).
    \end{eqnarray*}
For the second term we use (\ref{eq3}) and  (\ref{eq4}) to estimate
{\allowdisplaybreaks
\begin{eqnarray*} 
\lefteqn{
\mu(\{ x \in X : | (Tb)(x) | > \frac{\alpha}{2} \})
} \hspace{5mm} \\*
&= & \mu(\{ x \in X : 
   | \sum_i (SA_{r_i} b_i)(x) + \sum_i ((T - SA_{r_i} ) b_i)(x)   | > \frac{\alpha}{2} \})\\
&\le&  \mu( \{ x \in X : | (S\sum_i A_{r_i} b_i) (x) | > \frac{\alpha}{4} \} ) \\*
&&\hspace{2mm} {} + \mu( \{ x \in X : | \sum_i ((T - SA_{r_i} ) b_i)(x)   | > \frac{\alpha}{4} \} )\\
   &\le&  \frac{C (1+\delta)^d}{\alpha} \left( \| S \|_{q_0 \to q_0}^{q_0} + (1+ W) \right)  \| f \|_1.
\end{eqnarray*}
}
 We then conclude that $T$ is of weak type $(1,1)$ with a weak type estimate
 \begin{equation}\label{F1} 
 \| T \|_{L^1 \to L^{1,w}} 
\le C (1 + \delta)^d \left( 1+ W + \| T \|_{p_0 \to p_0}^{p_0} + \| S \|_{q_0 \to q_0}^{q_0}  \right).
\end{equation}
If we  replace $T$ and $S$ by $\| T \|_{p_0 \to p_0}^{-1} T$  and $\| T \|_{p_0 \to p_0}^{-1} S$ we  obtain (\ref{DMop}) with
 $\| T \|_{p_0 \to p_0}^{-1} W$ instead of $W$. Thus applying (\ref{F1}) to $\| T \|_{p_0 \to p_0}^{-1} T$, $\| T \|_{p_0 \to p_0}^{-1} S$
 and $\| T \|_{p_0 \to p_0}^{-1} W$ yields (\ref{L1}).
 
 Finally, by Marcinkiewicz interpolation theorem the operator $T$ extends to a bounded operator from $L^p(X) \cap L^{p_0}(X)$ to $L^p(X)$ for all
 $p \in (1, p_0)$. 
\end{proof}

 Following again an idea in  \cite{DuM3} we can prove a version of  the previous theorem  on arbitrary domains.
Let $\Omega$ be an open subset of $X$ and assume that $T$ is bounded
 on $L^{p_0}(\Om)$ and $S$ and  $A_t$ are bounded on $L^{q_0}(\Om)$.
We define $\widetilde T \colon L^{p_0}(X) \to L^{p_0}(X)$ by 
 $$\widetilde{T} f = \one_\Om T (\one_\Om f)$$
 and similarly for $\widetilde{S}$ and $\widetilde{A_t}$.
If $A_t$ satisfies (\ref{au2Om})  below then $\widetilde{A_t}$ satisfies 
 (\ref{au2}).
The operator 
 $T$ is weak type $(1,1)$ {\it if and only if} $\widetilde{T}$ is weak type $(1,1)$. 
Applying the previous theorem to $\widetilde{T}$, $\widetilde{S}$ and $\widetilde{A_t}$ gives 
the following result.

\begin{theorem}\label{th2}
Let  $T$ be  a non-zero bounded linear operator on $L^{p_0}(\Omega)$ for some $p_0 \in (1, \infty)$.
 Suppose there exists a bounded operator $S$ on $L^{q_0}(\Omega)$ for some $q_0 \in (1, \infty)$, 
 a family of
 bounded operators $(A_t)_{t > 0}$ on $L^{q_0}(\Omega)$ 
and a sequence $(g(j))_{j \in \Ni}$ in $\Ri$ 
such that 
\begin{equation}\label{au2Om}
 \left(\frac{1}{v(x,2^{j+1}t)}\int_{C_j(x,t) \cap \Omega} | A_{t}f |^{q_0} \right)^{1/q_0}
\le g(j) \frac{1}{v(x,t)}\int_{B(x,t) \cap \Omega} | f |
\end{equation}
for  all ball $x \in \Omega$, $t > 0$, $j \in \Ni$ and  $f \in L^{q_0}(B(x,t) \cap \Omega)$,
and  $\sum_{j=1}^\infty 2^{jd}g(j) < \infty$. 
Finally, suppose there exist $\delta,W > 0$ such that  
\begin{equation}\label{DM1op}
\int_{\Om \setminus B(x,(1+\delta)t )} | ((T  - SA_t) u)(y) | \, d\mu(y) \le W \| u \|_1
\end{equation} 
for all $x \in X$, $t > 0$ and  $u \in L^1(\Om) \cap L^\infty(\Om)$ supported  in the ball 
$B(x, t) \cap \Om$.  
Then $T$ is a weak type $(1,1)$ operator with
\begin{equation}\label{L1-om} 
\| T \|_{L^1(\Omega) \to L^{1,w}(\Omega)} 
\le C (1+\delta)^d \left( W + \| T \|_{p_0 \to p_0} 
   + \| S \|_{q_0 \to q_0}^{q_0}  \| T \|_{p_0 \to p_0}^{1-q_0} \right).
\end{equation}
Here $C$ is a constant depending only on the constants in {\rm (\ref{doubl2})}.
In particular, $T$ extends to a bounded operator on $L^p(\Omega)$ for all $p \in (1, p_0)$.
\end{theorem}

As in Remark~\ref{rpartmult202} the condition (\ref{DM1op}) follows if 
the operators $T$ and $SA_t$ are given by kernels $K$ and $K_t$ 
{\rm (}in the sense of {\rm (\ref{kernel}))} and there are $\delta,W > 0$
such that 
\begin{equation}\label{DM1}
\int_{\rho(x,y) \ge \delta t} | K(x,y) - K_t(x,y) | \, d\mu(x) \le W < \infty,
\end{equation} 
for all $t > 0$ and $y \in \Omega$. It suffices to note that the associated  kernel of $\widetilde{T}$ is the  extension by $0$ outside $\Om \times \Om$
 of the kernel of $T$ where $\widetilde{T} f = \one_\Om T (\one_\Om f)$ as above. Similarly for the kernel of $\widetilde{SA_t}$.

We may replace in the previous theorems the annulus
$C_j(x,r)$ by the annulus $A(x,j,r) := B(x,(j+1)r) \setminus B(x,jr)$.
In that case, $v(x, 2^{j+1}r)$ has to be replaced by
$v(x, (j+1)r)$ and the condition on $g$ becomes  $\sum_j  j^d g(j) < \infty$. 

Following   \cite{BK1} it is proved in   \cite{Aus2}   that a bounded operator $T$ on $L^2(X)$
is of weak type $(r,r)$ if 
\begin{equation}\label{au1}
\Big(\frac{1}{v(x, 2^{j+1}t)}\int_{C_j(x,t)} | T(I-A_{t})f |^2\Big)^{1/2}
\le g(j) \Big(\frac{1}{v(x,t)}\int_{B(x,t)} | f |^r\Big)^{1/r}
\end{equation}
and
\begin{equation}\label{au2'}
 \Big(\frac{1}{v(x,2^{j+1}t)}\int_{C_j(x,t)} |A_{t}f |^2\Big)^{1/2}
\le g(j) \Big(\frac{1}{v(x,t)}\int_{B(x,t)} | f |^r\Big)^{1/r}
\end{equation}
for all $x \in X$, $t > 0$, $j \in \Ni$ and $f \in L^2$ supported in $B(x,t)$,
and $\sum g(j)2^{dj}<\infty$.
One can prove a version of this result in which  $T-TA_t$ in (\ref{au1}) is replaced by 
$T-SA_t$ as in Theorem \ref{th1}.
We do not write the details here since we have no concrete application.
 Theorem \ref{th1} is suitable for our purpose. 

Finally, let us mention that a  Gaussian upper bound implies assumption  (\ref{au2}).
Indeed, assume that $A_t$ is given by a kernel $a_t$ such that
$$ | a_t(x,y) | \le \frac{C}{v(y,t^{1/m}) } \exp\{ - c \frac{\rho(x,y)^{m/(m-1)}}{t^{1/(m-1)}} \}$$
for all $t > 0$ and $x,y \in X$.
Here $m\ge 2$ and $C, c$ are two positive constants.
Fix $x_0 \in X$.
Note that the operator $\one_{C_j(x_0,t)} A_{t^m} \one_{B(x_0,t)}$ has the 
kernel
$(x,y) \mapsto \one_{C_j(x_0,t)}(x) a_{t^m}(x,y) \one_{B(x_0,t)}(y)$.
But 
\begin{eqnarray*}
\lefteqn{ |\one_{C_j(x_0,t)}(x) a_{t^m}(x,y) \one_{B(x_0,t)}(y)| 
} \hspace{30mm} \nonumber  \\*
& \leq & \frac{C}{v(y,t) }  \one_{C_j(x_0,t)}(x) \one_{B(x_0,t)}(y) 
      \exp\{ - c \frac{\rho(x,y)^{m/(m-1)}}{t^{m/(m-1)}}\}  \\
& \leq & \frac{C}{v(y,t) }  e^{-c 2^{j m/(m-1)}}  \one_{B(x_0,t)}(y)  \\
& \leq & \frac{C}{v(x_0,t) }  e^{-c 2^{j m/(m-1)}} ,
\end{eqnarray*}
where we used the doubling property in the last step.
This an $L^1$-$L^\infty$ estimate.
By interpolation, it  implies an $L^1$-$L^{q_0}$ estimate which gives  (\ref{au2}) 
for every $q_0 \in (1, \infty)$.

\section{A partial multiplier theorem for degenerate operators}

Let the coefficients $a_{kj}$, form $\fra_0$ and the self-adjoint operator $A$ 
associated with $\fra_0$ be as in the introduction.
For every bounded measurable function $F \colon  [0, \infty) \to \CC$,  
the operator  $F(A)$ is well defined by spectral theory and is bounded on $L^2(\RR^d)$. 
As  mentioned in the introduction, if $A$ is uniformly  elliptic  then $F(A)$ extends 
to a bounded operator on $L^p(\RR^d)$ for all $p \in (1, \infty)$ provided $F$ has a finite number 
of derivatives on $[0, \infty)$ which have good decay.
We address here the same problem for degenerate operators.
This is a difficult problem because no global Gaussian upper 
bounds are available for  $A$ in general.
We prove a partial result by projecting on the part where the  matrix $(a_{kj})$ is uniformly elliptic.
There are two versions.

 \begin{theorem}  \label{th3}
  Let $\Omega \subset \RR^d$ be an open bounded set with Lipschitz 
boundary.
Suppose there exists a $\mu > 0$ such that $(a_{kj}(x)) \geq \mu I$ for almost every
$x \in \Omega$ and denote by $P_\Omega$  the projection from $L^2(\RR^d)$ onto 
$L^2(\Omega)$.
 Set $H = A + I$. 
 Let $F\colon  [0, \infty) \to \CC$ be a bounded function such that 
 \begin{equation}\label{eq2.1-om}
 \sup_{t > 0} \| \varphi (.) F(t.) \|_{C^s} < \infty 
 \end{equation}
 for some $s > d/2$ and some non-trivial function  $\varphi \in C_{c}^\infty(0, 
 \infty)$.
 Then  $P_\Om F(H) P_\Om$  is of   weak type $(1,1)$   and 
 extends  to a  bounded operator  on $L^p(\RR^d)$ for all $p \in (1, \infty).$
 \end{theorem} 
 
 \begin{theorem}\label{th4}
  Let $\chi \in C_{\rm b}^\infty(\RR^d)$, $\mu > 0$ and 
suppose that $(a_{kj}(x)) \geq \mu I$ for almost every
$x \in \supp \chi$.
Set $H = A + I$. 
 Let $F\colon  [0, \infty) \to \CC$ be a bounded function such that 
 \begin{equation}\label{eq2.1-F}
 \sup_{t > 0} \| \varphi (.) F(t.) \|_{C^s} < \infty 
 \end{equation}
 for some $s > d/2$ and some non-trivial function  $\varphi \in C_{c}^\infty(0, 
 \infty)$.
 Then $M_\chi F(H)M_\chi $ is  of weak type $(1,1)$   and 
 extends  to a  bounded operator  on $L^p(\RR^d)$ for all $p \in (1, \infty).$
 \end{theorem} 

 The proof of both  theorems is almost the same.
It relies mainly on weighted estimates for the associated kernel of  
$M_\chi F(H)M_\chi $ (or the kernel of $P_\Om F(H) P_\Om$)  together with
 Theorem \ref{th1}.
The proof of weighted estimates for the kernel of $M_\chi F(H)M_\chi $ 
(or  of $P_\Om F(H) P_\Om$) is  based on partial Gaussian bounds proved in 
\cite{EO1} and a similar strategy as in  \cite{DOS} and \cite{Ouh5}. 
 
 In the sequel of this section we assume that there exists a constant
 $\mu > 0$ such that $(a_{kj}(x)) \geq \mu I$ for a.e.\ $x \in \Om$ 
respectively for a.e.\ 
$x \in \supp \chi \cup \supp \widetilde{\chi}$.
In the first case
 $\Om$ is a bounded Lipschitz domain of $\RR^d$ and in the second case 
$\chi , \widetilde{\chi} \in C_{\rm b}^\infty(\RR^d)$.
We   denote by $S_t := e^{-tA}$ the holomorphic 
 semigroup generated by $-A$ on $L^2(\RR^d)$.
 We recall the following result from \cite{EO1}. 

 \begin{theorem}\label{thEO} 
There are $C,c > 0$ such that for all $t > 0$ 
the operator $M_{\widetilde{\chi}}  S_t M_\chi$ 
{\rm (}respectively $P_\Om S_t P_\Om${\rm )} is given by a kernel $p_t$ which satisfies 
 $$ | p_t(x,y) | \le C t^{-d/2} e^{-c \frac{|x-y|^2}{t}} (1 + t)^{d/2} 
\mbox{ for all } t > 0 \mbox{ and }  x, y  \in \RR^d.$$
 \end{theorem}
  The theorem is stated in \cite{EO1} with $\chi = \widetilde{\chi}$ but the arguments work with different $\chi$ and $\widetilde{\chi}$.
It is also proved  there  that 
\begin{equation}\label{1-2}
 \| M_\chi S_t \|_{2 \to \infty} \le Ct^{-d/4} (1+ t)^{d/4}
\quad\! \mbox{respectively} \quad\!
   \| P_\Om S_t \|_{2 \to \infty} \le Ct^{-d/4} (1+t)^{d/4}.
 \end{equation}
If $z = t+ is \in \CC $ with $t = \R  z > 0$, then 
\begin{eqnarray*}
\| M_\chi  S_z M_\chi \|_{1 \to  \infty} &=& \| M_\chi  S_{t/2} S_{is} S_{t/2} M_\chi \|_{1 \to \infty} \\
&\le & C t^{-d/4} (1+ t)^{d/4}  \| S_{is} S_{t/2} M_\chi \|_{1 \to 2} \\
&\le& C t^{-d/4}  (1+ t)^{d/4} \| S_{t/2} M_\chi \|_{2 \to \infty} \\
& \le & Ct^{-d/2} (1+ t)^{d/2}.
\end{eqnarray*}
Similarly, 
\begin{equation}\label{1-22}
 \| P_\Om S_z P_\Om \|_{1 \to \infty} \le C (\R z)^{-d/2} (1+ \R z)^{d/2}
 \end{equation}
for all $z \in \CC$ with $\R  z > 0$.
Using the Gaussian bounds of Theorem~\ref{thEO} for real $t$ together with the 
uniform bounds (\ref{1-22}) for complex $z$ 
it follows as in Theorem 3.4.8 in \cite{Dav2} or Theorem 7.2 in \cite{Ouh5} 
that for all $\varepsilon > 0$ the kernel $p^{(0)}_z$ of 
$M_\chi S_z M_\chi$ respectively $P_\Om S_z P_\Om$ satisfies the bound 
\begin{equation}\label{eq2.2}
| p^{(0)}_z(x,y) e^{- \eps z}  | 
\le C_\eps (\R z)^{-d/2} \exp\{ - c \frac{| x - y |^2}{|z|} \cos(\arg z) \} 
\end{equation} 
for all $x,y \in \Ri^d$ and $z \in \CC$ with $\RRe z > 0$.

Let  $H = A + I$ and define $p_z(x,y) = p^{(0)}_z(x,y) e^{-z}$.
Then $p_z$ is the kernel of 
$M_\chi e^{-zH} M_\chi$.
 We shall formulate the results below for $M_\chi F(H) M_\chi$ 
only, but all statements are also valid  for  $P_\Om F(H) P_\Om$. 
In the following lemmas, we shall
always assume that $(a_{kj}(x)) \geq \mu I$ for almost every
$x \in \supp \chi$.
Since associated kernels with several operators are involved in the sequel  
we shall  denote by $K_T$ the  kernel associated to a 
given operator $T$ whenever it exists.

 \begin{lemma}\label{le1} 
For all $s > 0$ and $\varepsilon > 0$ there exists a  $C > 0$ such that 
 $$\int_{\RR^d} | K_{M_\chi F(H) M_\chi}(x,y) |^2 \, (1 + \sqrt{r} |x-y|)^s \, dx 
\le C r^{d/2} \| \delta_r F \|_{C^{s/2 + \eps}}^2$$
 for all $r > 0$, $y \in \Ri^d$ and  $F \in C^{s/2 + \eps}$ supported in $[0,r]$.
 Here $(\delta_rF)(\la) := F(r \la)$. 
 \end{lemma}
 \begin{proof} The arguments are very  similar to those of Lemma 4.3  in \cite{DOS}.
Fix $r > 0$ and assume first that $F$ is supported in $[0,1]$. 
Set $g(\la) := F(\la) e^{\la}$ and $H_r := \frac{1}{r} H$.
By (\ref{eq2.2}), the kernel $p_{z/r}$ of $M_\chi e^{-zH_r} M_\chi$ satisfies
\begin{equation}\label{eq2.3}
|p_{z/r}(x,y)| 
\le C r^{d/2} (\R z)^{-d/2} \exp\{ - c r  \frac{| x - y |^2}{|z|} \cos(\arg z) \} 
\end{equation} 
for all $x,y \in \Ri^d$ and $z \in \CC$ with $\RRe z > 0$,
with constants $C, c $ independent of $r$.
 We write 
$g(\la) = \int_\RR \hat{g}(\xi) e^{i\la \xi} \, d\xi$, where $\hat{g}$ is 
the Fourier transform of $g$.
Then
$$F(H_r) = \int_\RR \hat{g}(\xi) \, e^{-(1 - i\xi)H_r} \, d\xi,$$
from which one obtains
\begin{equation}\label{eq2.4}
K_{M_\chi F(H_r) M_\chi}(x,y) = \int_\RR \hat{g}(\xi) \, p_{(1-i\xi)/r}(x,y) \, d\xi.
\end{equation}
Let $y \in \RR^d$.
Using the estimate (\ref{eq2.3}) with $z = 1 - i \xi$ gives
\begin{eqnarray*}
\lefteqn{
\int _{\RR^d} | p_{(1-i\xi)/r}(x,y) |^2 \, (1 + \sqrt{r}|x-y|)^s \, dx
} \hspace{30mm} \\*
 &\le& C r^d \int \exp\{ - 2 c r  \frac{| x - y |^2}{1 + \xi^2}  \} \, (1 + \sqrt{r}|x-y|)^s \, dx\\
&\le& C r^d (1 + \xi^2)^{s/2} \int \exp\{ - c r  \frac{| x - y |^2}{1 + \xi^2}  \} \, dx \\
&\le& C r^d (1 + \xi^2)^{s/2} \, \Big(\frac{1 + \xi^2}{r} \Big)^{d/2} \\
&=& C r^{d/2} (1 + \xi^2)^{(d+s)/2}.
\end{eqnarray*}
It follows from (\ref{eq2.4}), the continuous version of the Minkowski inequality
and the previous estimate that 
\begin{eqnarray}
\lefteqn{
\left( \int_{\RR^d}   | K_{M_\chi F(H_r) M_\chi}(x,y) |^2 \, 
    (1 + \sqrt{r} |x-y|)^s \, dx \right)^{1/2} 
} \hspace{20mm} \nonumber  \\*
&\le& \int_\RR |\hat{g}(\xi)| \left(  \int _{\RR^d} | p_{(1-i\xi)/r}(x,y) |^2 \, 
(1 + \sqrt{r}|x-y|)^s dx \right)^{1/2} \, d\xi \nonumber \\
&\le& C r^{d/4} \int_\RR | \hat{g}(\xi )| \, (1 + \xi^2)^{(d+s)/4 } \, d\xi \nonumber \\
& \le & C r^{d/4} \| g \|_{W^{(d+s+2)/2 ,2}} \nonumber \\
& \le & C r^{d/4} \|  F \|_{W^{s/2 + \alpha ,2}}. \label{eq2.5}
\end{eqnarray}
Here $\alpha=  (d+2)/2$ and the constants are independent of $r$ and $y$.
On the other hand $M_\chi F(H_r) M_\chi = M_\chi g(H_r) e^{-H_r} M_\chi$.
It follows from (\ref{1-2}) that 
\[
\|e^{-H_r} M_\chi \|_{1 \to 2}
\leq C r^{d/4} (1 + \tfrac{1}{r})^{d/4} \, e^{-1/r}
\leq C r^{d/4}
\]
for all $r > 0$.
Moreover, $\|M_\chi g(H_r)\|_{2 \to 2} \leq e \|\chi\|_\infty \|F\|_\infty$.
Therefore 
\begin{equation}\label{eq2.6}
\int_{\RR^d}   | K_{M_\chi F(H_r) M_\chi}(x,y) |^2  \, dx 
\leq \|M_\chi F(H_r) M_\chi\|_{1 \to 2}^2
\le C r^{d/2} \| F \|_\infty^2.
\end{equation}
This is valid for all $F$ with support in $[0,1]$ and for all $s > 0$.
The estimates (\ref{eq2.5}) and (\ref{eq2.6}) together with an  
interpolation argument  (see \cite{MauM}, p.\ 151  and \cite{DOS}, p.\ 455) give
then that for all $s > 0$ there exists a $C > 0$ such that 
\begin{equation}\label{eq2.7}
\int_{\RR^d}   | K_{M_\chi F(H_r) M_\chi}(x,y) |^2(1 + \sqrt{r} |x-y|)^s dx 
\le C r^{d/2} \| F \|_{C^{s/2 + \eps}}^2.
\end{equation}
Finally, if $F$ has support in $[0,r]$ we use the last estimate with $\delta_rF$ and obtain the lemma. 
 \end{proof}

 \begin{lemma}\label{le2}
 The operators $A_t :=  e^{-t^2 H}M_\chi$ satisfy {\rm (\ref{au2})}. 
 \end{lemma}
 \begin{proof}
Let $\psi \in W^{1,\infty}(\Ri^d,\Ri)$ be such that $| \nabla \psi | \le 1$.
For all $\rho \in \Ri$ define $U_\rho = M_{e^{\rho \psi}}$
and set  
 $S_t^\rho := U_\rho e^{-tH}  U_{-\rho} $.
It follows from \cite{EO1} Proposition~3.6 by duality and a limit $n \to \infty$ 
that there exist $C,\omega > 0$, independent of $t$, $\rho$ and $\psi$, 
such that 
 \begin{equation}\label{eq2.8}
 \| S_t^\rho M_\chi \|_{1 \to 2} \le C t^{-d/4} e^{\omega \rho^2 t}.
 \end{equation}
Now fix two bounded open non-empty sets $E$ and $F$ of $\RR^d$ and choose  
$\psi(x) := d(x, E) \wedge N$, where $N = \sup \{ |x-y| : x \in E, \; y \in F \}  + 1$.
For all $h \in L^2(E)$ and $\rho \geq 0$ one has 
$$M_\chi e^{-tH} h = M_\chi U_{-\rho} S_t^\rho h.$$
Therefore
\[
\| M_\chi e^{-tH} h \|_{L^\infty(F)} 
\le e^{-\rho d(E,F)} \| M_\chi S_t^\rho h \|_\infty
\le C t^{-d/4} e^{-\rho d(E,F)} e^{\omega \rho^2 t} \| h \|_2.
\]
Choosing $\rho = \frac{d(E,F)}{2 \omega t}$ yields the  Davies--Gaffney type estimate
\begin{equation}\label{2.9}
\|P_F  (M_\chi e^{-t H}) P_E \|_{2 \to \infty} 
\le C t^{-d/4} e^{-\frac{d(E,F)^2}{4 \omega t}}.
\end{equation}
In particular,
$$ \|  P_{C_j(x,t)} e^{-t^2 H} M_\chi P_{B(x,t)} \|_{1 \to 2} \le C t^{-d/2} e^{-c4^j}$$
for all $x,y \in \Ri^d$ and $j \in \Ni$.
This shows the lemma.
\end{proof}

 \begin{proof}[Proof of Theorems \ref{th3} and \ref{th4}]
 As mentioned above, the proof of both theorems is almost the same.
We consider $M_\chi F(H) M_\chi$ only.  
The proof is based on Theorem \ref{th1} and the previous lemmas.
It is in the same spirit as in the elliptic case where a Gaussian bound holds (cf.\ \cite{DOS}, \cite{Ouh5}).
 Let  $\varphi \in C_c^\infty(0, \infty)$ be such that $\supp \varphi \subset [1/4, 1]$  
and 
$$\sum_{n=-\infty}^{\infty} \varphi(2^{-n}\la) = 1 $$
for all $\lambda > 0$.
Then 
$$F(\la) = \sum_{n=-\infty}^{\infty} \varphi(2^{-n}\la) F(\la) 
=: \sum_{n=-\infty}^{\infty} F_{n}(\la).$$
We apply Theorem \ref{th1} to   $M_\chi F_n(H) M_\chi$ for each fixed $n \in \ZZ$.
 We choose 
$$ S := M_\chi F_n(H)
\quad \mbox{and} \quad
A_t:=  e^{-t^2H} M_\chi.$$
By Lemma \ref{le2}, the operators $A_t$ satisfy (\ref{au2}).
It remains to prove (\ref{DM}). 
For this we  have to estimate for all $y \in \RR^d$ the integral
$$I_{n,t} := \int_{| x - y |  \ge t}  | K_{M_\chi  G_{n,t}(H)  M_\chi}(x,y) | \, dx,$$
where 
$G_{n,t} (\la) = F_n(\la)   -  F_n(\la)  e^{-t^2 \la} 
= \varphi(2^{-n} \lambda) F(\la) (1- e^{-t^2 \la})$. 
First, by the Cauchy--Schwarz inequality we have  
\begin{eqnarray}
 I_{n,t} 
&\le& \Big( \int_{\RR^d} | K_{M_\chi   G_{n,t}(H) M_\chi}(x,y) |^{2}
 \, (1 + 2^{n/2} | x - y |)^{2s} \, dx \Big)^{1/2} \times \nonumber \\
\label{2.10}  &&  \hspace*{50mm} {}
\times  \Big( \int_{| x - y |  \ge t} 
       (1 +  2^{n/2} | x - y | )^{-2s} \, dx \Big)^{1/2} 
\end{eqnarray}
We apply Lemma \ref{le1} with  $r = 2^n$ and obtain 
\begin{equation}\label{2.11}
 \int_{\RR^d} | K_{M_\chi G_{n,t}(H) M_\chi}(x,y)|^{2} \, (1 + 2^{n/2} |x-y| )^{2s} \,  dx 
\le C 2^{nd/2} \| \delta_{2^n} G_{n,t} \|_{C^{s+\eps}}^{2}.
\end{equation}
Simple computations 
show that there exists a $C > 0$, independent of $n$
and $t$,  such that
\begin{eqnarray}
\| \delta_{2^n} G_{n,t} \|_{C^{s+\eps}} &=& \| \varphi(.) F(2^n .) (1-e^{-t^2 2^n .} )  \|_{C^{s+\eps}} \nonumber\\
&\le& C \sup_{t'>0}\| \varphi(.) F(t' .) \|_{C^{s+\eps}} \min (1, t^2 2^n). \label{2.12} 
\end{eqnarray}
On the other hand (see \cite{DOS} or (7.46) in \cite{Ouh5}) one estimates
\begin{equation}\label{2.13} 
\int_{| x-y |  \ge t} (1 + 2^{n/2} |x-y|)^{-2s} \, dx
\le C 2^{-nd/2}  \min (1, (t2^{n/2})^{d-2s}). 
\end{equation}
Using (\ref{2.10}), (\ref{2.11}), (\ref{2.12}) and (\ref{2.13}) we obtain
$$I_{n,t} \le C  \min (1, t^2 2^n) \min (1, (t 2^{n/2})^{\frac{d}{2} - s}) 
\sup_{t'>0}\| \varphi(.) F(t' .) \|_{C^{s+\eps}}.$$
Hence 
$$\sum_{n=-\infty}^{\infty} I_{n,t} 
\le C \left( \sum_{n \in \ZZ, \; t^2 2^n \le 1} t^2 2^n 
    + \sum_{n \in \ZZ, \; t2^{n/2} > 1} (t2^{n/2})^{\frac{d}{2}-s} \right) 
   \sup_{t'>0}\| \varphi(.) F(t' .) \|_{C^{s+\eps}}$$
and the RHS is bounded by a constant independent  $t$ since $s > \frac{d}{2}$.
 This proves Theorem~\ref{th3}.
\end{proof}
 
 As explained in the introduction, the reason why we consider $H = A + I$ instead of $A$ in the previous results comes from  
the fact the Gaussian upper bound in Theorem \ref{thEO} is valid with
the extra factor $(1+ t)^{d/2}$.
 If one considers the case where 
 $a_{kj} = \delta_{kj}$ on a smooth bounded domain $\Om$, then $A$ is the Neumann Laplacian on $L^2(\Om)$ and $0$ on $L^2(\RR^d\setminus \Om)$.
It is then easy to see 
 that $L^2$-$L^\infty$ estimates (respectively, Gaussian bounds) for 
$M_\chi e^{-tA}$ or $P_\Om e^{-tA}$  (respectively, $M_\chi e^{-tA}M_\chi$ 
or $P_\Om e^{-tA}P_\Om$)
 cannot hold without an extra factor  $(1+t)^{d/4}$ (respectively, $(1+t)^{d/2}$).
On the other hand we can replace in the previous theorems $H = A + I$ by $H = A + \eps I$ for any $\eps > 0$. 

It may be possible that if $(a_{kj}) \ge \mu I$ on a connected subset $F$ of 
$\RR^d$ which is `large enough' (in some sense) one obtains Theorem \ref{thEO} 
without the extra factor $(1 + t)^{d/2}$ in the Gaussian bound.
This  remains to be proved.
We mention that if such bound holds, we obtain by the same proof  
Theorems \ref{th3} and \ref{th4} for  $F(A)$ rather than  $F(H)$. 

We emphasize also that we consider here general degenerate operators with non-smooth coefficients.  
One may obtain global results for some specific operators which are degenerate at 
every point and the coefficients are not continuous at every point.
For example, one might take a pure second-order subelliptic operator
in divergence form with real measurable coefficients on a Lie group
with polynomial growth.
Then global Gaussian bounds are valid by \cite{SS}, Th\'eor\`eme~1 together
with regularization argument (see, for example Section~2.1 in \cite{ER15}).
Therefore a global spectral multiplier result  for such operators follows directly from 
\cite{DOS}.  
Note however that  the order of smoothness required on the function $F$ is 
larger than half of  Euclidean dimension. 
On the other hand, the operators that we consider in this paper 
are allowed to vanish on big sets.

\begin{examples}
 We give some examples which are direct applications of the previous theorems. 
\noindent
\begin{description}
\item[{\it Imaginary powers.}]  
Set  $F(\la) = \la^{is}$ where $s \in \RR$.
Then Theorems \ref{th3} and \ref{th4} together with the Riesz--Thorin interpolation theorem  
imply that 
for all $\varepsilon > 0$ and $p \in (1,\infty)$ there exists a $C > 0$ such that
$$ \| M_\chi H^{is} M_\chi \|_{p \to p} \le C_\eps (1 + | s | )^{(d + \eps)| \frac{1}{2} - \frac{1}{p} |}$$
and
$$ \| P_\Om H^{is} P_\Om \|_{{\mathcal L}(L^p)} \le C_\eps (1 + | s | )^{(d + \eps)| \frac{1}{2} - \frac{1}{p} |}$$
for all $s \in \RR$. 

\item[{\it The Schr\"odinger group.}]
 Set $F(\la) = (1+ \la)^{-\alpha} e^{it \la}$ with $t \in \RR$ and $\alpha > d/2$.
 The operators  $M_\chi(I + H)^{-\alpha} e^{itH}M_\chi$ and 
$P_\Om (I + H)^{-\alpha} e^{itH} P_\Om$ are bounded on $L^p(\RR^d)$ for all $p \in (1, \infty)$.
Their  $L^p$-norms are estimated  by $C(1+ | t |)^\alpha$.
By interpolation, we obtain 
boundedness on $L^p$ for all $\alpha >  d |\frac{1}{2} - \frac{1}{p}|$ and 
$t \in \Ri$.

Remark. 
Using the same proof  as in \cite{CCO}, these results can be obtained directly 
from the Gaussian upper bound of Theorem \ref{thEO}  without appealing to 
Theorems \ref{th3} and \ref{th4}.

\item[{\it Wave operators.}] 
Set $F(\la) = (1+ \la)^{-\alpha/2} e^{it \sqrt{\la}}$ with $t \in \RR$ and 
$\alpha > d/2$.
The operators  $M_\chi(I + H)^{-\alpha} e^{it\sqrt{H}}M_\chi$ and 
$P_\Om (I + H)^{-\alpha} e^{it\sqrt{H}} P_\Om$ are bounded on $L^p(\RR^d)$ for all 
$p \in (1, \infty)$. 
\end{description}
\end{examples}

\section{Riesz transforms}
The aim in this section is to prove boundedness on $L^p(\RR^d)$ of a type of  
Riesz transform  operator  $M_\chi \nabla (I + A)^{-1/2} M_\chi$.  
We keep the same notation as in the previous section.  
The main result of this section is the next theorem.

\begin{theorem}\label{th7} 
Let $\chi \in C_{\rm b}^\infty(\RR^d)$, $\mu > 0$ and 
suppose that $(a_{kj}(x)) \geq \mu I$ for almost every
$x \in \supp \chi$.
Set $H = A + I$.
Then for every $k \in \{1, \ldots, d\}$,  the operator $M_\chi  \partial_k H^{-1/2} M_\chi$ is of weak 
type $(1,1)$ and is bounded on $L^p(\RR^d)$ for all $p \in (1, 2]$. 
\end{theorem}

Here $\partial_k$ denotes the distributional derivative.
The proof is based on  Theorem \ref{th1} and uses some ideas  from \cite{CoD}, \cite{DuM2} and  Chapter 7 in \cite{Ouh5} in the uniformly elliptic case. 
We start with the following lemma. 
Let $\fra$ be the closure of 
the regular part of the form $\fra_0$
defined in (\ref{formcc}).

\begin{lemma}\label{lemm1} 
Let $\chi \in C_{\rm b}^\infty(\RR^d)$, $\mu > 0$ and 
suppose that $(a_{kj}(x)) \geq \mu I$ for almost every
$x \in \supp \chi$.
Then $\chi u \in W^{1,2}(\Ri^d)$ and 
\[
\|\chi \partial_k u\|_2^2 
\leq \frac{\|\chi\|_\infty^2}{\mu} \, \|H^{1/2} u\|_2
\]
for all $u \in D(\fra) = D(H^{1/2})$ and $k \in \{ 1,\ldots,d \} $. 
\end{lemma}
\begin{proof}  
Let $u \in D(\fra)$.
Then there exists a sequence $(u_n)_{n \in \NN}$ in $D(\fra_0) = C_c^\infty(\RR^d)$
such that $\lim u_n = u$ in $L^2(\RR^d)$ and 
$\fra(u) = \lim \fra_0(u_n)$.
By the ellipticity assumption on the support of $\chi$ one deduces
\begin{equation} \label{elemm1;1}
\mu \int_{\RR^d} \chi^2 \, |\nabla u_n|^2
\leq \sum_{k,j=1}^d \int_{\RR^d} a_{kj} \, (\partial_k u_n) \, (\partial_j u_n) \, \chi^2
\leq \|\chi\|_\infty^2 \, \fra_0(u_n)
\end{equation}
for all $n \in \NN$.
Therefore $(\chi u_n)_{n \in \NN}$ is bounded in $W^{1,2}(\RR^d)$.
Hence it has a weakly convergent subsequence in $W^{1,2}(\RR^d)$.
Since $\lim \chi u_n = \chi u$ in $L^2(\RR^d)$ it follows that 
$\chi u \in W^{1,2}(\RR^d)$.
Then taking the limit $n \to \infty$ in (\ref{elemm1;1}) one estimates
\[
\mu \int_{\RR^d} \chi^2 \, |\partial_k u |^2
\leq \|\chi\|_\infty^2 \, \fra(u)
\leq \|\chi\|_\infty^2 \, \|H^{1/2} u\|_2^2
\]
for all $k \in \{ 1,\ldots,d \} $.
\end{proof}

\begin{lemma}\label{lemm2}
Let $\chi, \widetilde{\chi} \in C_{\rm b}^\infty(\RR^d)$,
$\mu > 0$ and assume that $(a_{kj}(x)) \geq \mu I$ for almost every
$x \in \supp \chi \cup \supp \widetilde{\chi}$.
Then for all $\beta > 0$  small enough we have 
$$\int_{\RR^d} | (M_{\widetilde{\chi}} e^{-sH} M_\chi u)(y) |^2 \, e^{\beta \frac{|x-y|^2}{s}} \, dy 
\le C s^{-d/2} e^{2\beta t^2/s} e^{-s} \| u \|_1^2$$
for  all $t > 0$, $s > 0$, $x \in \RR^d$ and $u \in L^2(\RR^d)$ with $\supp u \subset B(x,t)$.
\end{lemma}
\begin{proof} 
By Theorem \ref{thEO} one estimates
\begin{eqnarray*}
\lefteqn{| (M_{\widetilde{\chi}} e^{-sH} M_\chi u)(y) |^2 \, e^{\beta \frac{|x-y|^2}{s}}
} \hspace{30mm} \\*
&=& | \int_{B(x,t)} e^{-s} p_s(y,z) \, u(z) \, dz |^2 \, e^{\beta \frac{|x-y|^2}{s}}\\
&\le& C e^{-s}  \Big(  \int_{B(x,t)} s^{-d/2} \, e^{-c \frac{|y-z|^2}{s}} \, 
    e^{\beta \frac{|x-y|^2}{2s}} \, |u(z)| \, dz \Big)^2\\
&\le&  C e^{-s}  \Big(  \int_{\RR^d} s^{-d/2} \, e^{-(c-\beta) \, \frac{|y-z|^2}{s}} 
    \,  |u(z)| \, dz \Big)^2 e^{2 \beta t^2/s}  \\
&\le&  C e^{-s} s^{-d/2} \, \|u\|_1  \int_{\RR^d} s^{-d/2} \, e^{-(c-\beta) \, \frac{|y-z|^2}{s}} 
    \,  |u(z)| \, dz \, e^{2 \beta t^2/s}. 
\end{eqnarray*}
Taking $\beta < \frac{1}{2} c$ and integrating over $y$ yields the lemma. 
\end{proof}

Since $e^{-sH}L^2(\RR^d) \subset D(\fra)$ for all $s > 0$ we obtain from Lemma \ref{lemm1} 
the inclusion 
$M_\chi \nabla e^{-sH}M_\chi( L^2(\RR^d)) \subset W^{1,2}(\RR^d)$ for all $s > 0$.  
The following weighted $L^2$-estimate  is in the same spirit as weighted gradient estimates for heat kernels (see 
\cite{CoD}, \cite{Gri2}  and Theorem 6.19 in \cite{Ouh5}).

\begin{lemma}\label{lemm3} 
For all $\beta > 0$ small enough we have 
$$ \int_{\RR^d} | (M_\chi \nabla e^{-sH} M_\chi u)(y) |^2 \, e^{\beta \frac{| x-y|^2}{s}} \, dy  
\le C s^{-d/2 -1} \, e^{6 \beta t^2/s} \, \| u \|_1^2$$
for  all $t > 0$, $s > 0$, $x \in \RR^d$ and $u \in L^2(\RR^d)$ with $\supp u \subset  B(x,t)$.
\end{lemma}
\begin{proof} 
In order to avoid domain problems of forms we shall proceed by approximation.
First, we prove the lemma for uniformly elliptic 
coefficients with constants $\beta$ and $C$ depending only on 
$\mu > 0 $  such that  $(a_{kj}(x)) \ge \mu \, I$ a.e.\  $x \in \supp \chi$. 

Assume that there exists a $\mu_0 > 0$ such that $(a_{kj}(x)) \ge \mu_0 \, I$ for 
a.e.\ $x \in \RR^d$. 
In this case the form $\fra$ has domain $W^{1,2}(\RR^d)$.
We use similar ideas as in the  proof of Theorem 6.19 in \cite{Ouh5} 
but we want to prove that the constants  in the estimates are independent of $\mu_0$. 
Let $\psi \in C_c^\infty(\Ri^d)$ be such that $\psi(x) = 1$ for all $x \in B(0,1)$
and $0 \leq \psi \leq \one$.
For all $n \in \Ni$ define $\psi_n \in C_c^\infty(\Ri^d)$ by 
$\psi_n(x) = \psi(\frac{1}{n} \, x)$.
Set  
$$ I_n :=  \int_{\RR^d} | \chi(y) \, (\nabla   e^{-sH} M_\chi u)(y) |^2 
   \, e^{\beta \frac{| x-y|^2}{s}} \,  \psi_n(y) \, dy$$
 and  define $ f  :=  e^{-sH} M_\chi u$. 
Then, 
\begin{eqnarray*}
I_n 
&\le& \frac{1}{\mu} \sum_{k,j} \int_{\RR^d} a_{kj}(y) \, (\partial_k f)(y) \, 
    (\partial_j f)(y) \, e^{\beta \frac{| x-y|^2}{s}} \, \chi(y)^2 \,  \psi_n(y) \, dy\\
&=& \frac{1}{\mu} \sum_{k,j} \int_{\RR^d} a_{kj} \, (\partial_k f) \, 
    \partial_j \Big( f \, e^{\beta \frac{| x- \cdot|^2}{s}} \, \chi^2 \, \psi_n \Big) \\*
&& \hspace{3mm} {}+ \frac{1}{\mu} \sum_{k,j} \int_{\RR^d} a_{kj}(y) \, (\partial_k f)(y) \,
     f(y) \, \frac{2 \beta (x_j-y_j)}{s} \, e^{\beta \frac{| x-y|^2}{s}} \, \chi(y)^2 \,
       \psi_n(y) \, dy\\*
&& \hspace{3mm} {}- \frac{2}{\mu} \sum_{k,j} \int_{\RR^d} a_{kj}(y) \, (\partial_k f)(y) \,
     f(y) \, e^{\beta \frac{| x-y|^2}{s}} \, (\partial_j \chi)(y) \, \chi(y) \, \psi_n(y) \, dy\\*
&& \hspace{3mm} {}- \frac{1}{n \mu} \sum_{k,j} \int_{\RR^d} a_{kj}(y) \, (\partial_k f)(y) \, f(y) \,
     e^{\beta \frac{| x-y|^2}{s}} \, \chi(y)^2  \, (\partial_j\psi)(\tfrac{1}{n} \, y) \, dy\\
& =: & J_{1,n} + J_{2,n} + J_{3,n} + J_{4,n}.
\end{eqnarray*}
Since $y \mapsto f (y) e^{\beta \frac{| x-y|^2}{s}} \chi(y)^2 \psi_n(y)$ 
is an element of $W^{1,2}(\RR^d)$ we have 
\begin{eqnarray*}
J_{1,n} 
&=& \frac{1}{\mu} \, \fra(  f,  f e^{\beta \frac{|x - \cdot|^2}{s}}  \chi^2 \psi_n )\\
&=& \frac{1}{\mu} \, \int_{\RR^d} (A e^{-sH} M_\chi u) \, (e^{-sH} M_\chi u) \, 
      e^{\beta \frac{| x- \cdot|^2}{s}} \, \chi^2 \, \psi_n \\
&\le &  \frac{\| \chi \|_\infty}{\mu} \,
    \| H e^{-sH} M_\chi u \|_2 \, 
     \| e^{\beta \frac{|x - \cdot|^2}{s}} \, M_\chi e^{-sH} M_\chi u \|_2 .
\end{eqnarray*}
The standard estimate $\| H e^{-sH} \|_{2 \to 2} \le s^{-1} $ and Lemma \ref{lemm2}  give 
\begin{equation}\label{J1}
J_{1,n}  \le C  s^{-d/2 -1} e^{2\beta t^2/s} \| u \|_1^2
\end{equation} 
if $\beta$ is small enough.
Using the obvious inequality 
$\frac{|x_j-y_j|}{s}  \le \frac{1}{\sqrt{\varepsilon \, s}} \, e^{\eps \frac{| x - y |^2}{s}} $ 
 we have 
\begin{eqnarray*}
|J_{2,n}|  
&\le&  \frac{C}{\sqrt{s}} \sum_k \int_{\RR^d}  
   |\partial_k e^{-sH} M_\chi u | \, \chi^2 \, |e^{-sH} M_\chi u| \,
     e^{2\beta \frac{| x-\cdot|^2}{s}}\\
&\le& \frac{C}{\sqrt{s}}  \sqrt{I_n}  
   \Big( \int_{\RR^d} | (M_\chi e^{-sH} M_\chi u)(y) |^2 e^{3\beta \frac{| x-y|^2}{s}} \, dy \Big)^{1/2}.
\end{eqnarray*}
Therefore Lemma \ref{lemm2} implies
\begin{equation}\label{J2}
|J_{2,n}|   \le C \sqrt{I_n} s^{-\frac{d}{4} - \frac{1}{2}} e^{-s/2} e^{3\beta t^2/s} \| u \|_1.
\end{equation} 
We estimate the third term in a similar way. 
\begin{eqnarray}
| J_{3,n} | 
&\le& C  \sum_k \int_{\RR^d} 
 |\partial_k e^{-sH} M_\chi u | \, |\chi \, \partial_j \chi | \,
     |e^{-sH} M_\chi u| \, e^{\beta \frac{| x- \cdot|^2}{s}} \, \psi_n \nonumber  \\
&\le& C \, \sqrt{I_n} 
   \Big( \int_{\RR^d} | M_{\partial_j\chi} e^{-sH} M_\chi u |^2 \, 
         e^{\beta \frac{| x- \cdot|^2}{s}} \Big)^{1/2} \nonumber  \\
&\le& C \sqrt{I_n} s^{-\frac{d}{4} - \frac{1}{2}} \,  e^{-s/3} \, e^{\beta t^2/s} \, \| u \|_1.
\label{J3}
\end{eqnarray}
Finally,
\begin{eqnarray}
| J_{4,n} | 
&\le& \frac{C}{n}   \sum_{k,j}  \int_{\RR^d} 
 |  (\chi \partial_k  e^{-sH} M_\chi u)(y) | \, | (M_\chi e^{-sH} M_\chi u)(y) | \,
     e^{\beta \frac{| x-y|^2}{s}} \nonumber  \times \\*
 && \hspace{75mm} {} \times | (\partial_j \psi)(\tfrac{1}{n} \, y)| \, dy  \nonumber   \\
&\le& \frac{C}{n}  \| M_\chi \, \nabla \, e^{-sH} \, M_\chi u \|_2 
    \Big( \int_{\RR^d}   | (M_\chi e^{-sH} M_\chi u)(y) |^2 \, 
      e^{2\beta \frac{| x-y|^2}{s}} \, dy \Big)^{1/2}  \nonumber   \\
&\le& \frac{C}{n}  \| M_\chi \, \nabla \, e^{-sH} \, M_\chi u \|_2 \,
    s^{-d/4} e^{-s/2} e^{2\beta t^2/s} \| u \|_1.
\label{J4}
\end{eqnarray}
Therefore, we obtain from (\ref{J1}), (\ref{J2}), (\ref{J3}) and (\ref{J4}) that
$$ I_n  \le C s^{-d/2-1} e^{6\beta t^2/s} \| u \|_1^2 
   +  \frac{C}{n} \, \| M_\chi \, \nabla \, e^{-sH} \, M_\chi u \|_2 \, 
    s^{-d/4}  e^{2\beta t^2/s} \| u \|_1.$$
Letting $n \to \infty$ and then  use  Fatou's lemma yields 
\begin{equation}\label{Ifinal}
 \int_{\RR^d} | (M_\chi \nabla e^{-sH} M_\chi u)(y) |^2 \, e^{\beta \frac{| x-y|^2}{s}} \, dy  
\le C s^{-d/2 -1} \, e^{6 \beta t^2/s} \| u \|_1^2.
\end{equation}
The constants $C$ and $\beta$ are independent of $\mu_0$. 

Now we prove the lemma for degenerate operators. 
For all $n \in \Ni$ set $a_{kj}^{(n)} = a_{kj} + \frac{1}{n} \delta_{kj}$.
Then $(a_{kj}^{(n)}(x)) \ge \frac{1}{n} \, I$ for a.e.\  $x \in \RR^d$
and $(a_{kj}^{(n)}(x)) \ge  \mu \, I$ for a.e.\ $x \in \supp \chi$.
Moreover, $\| a_{kj}^{(n)} \|_\infty \le 1 + \| a_{kj}  \|_\infty$.
 We denote by $A_n$ 
the elliptic operator with the coefficients $a_{kj}^{(n)}$ and $H_n = I + A_n$.
 We apply (\ref{Ifinal}) to $H_n$  and obtain 
\begin{equation}\label{final-n}
 \int_{\RR^d} | (M_\chi \nabla e^{-sH_n} M_\chi u)(y) |^2 \, 
    e^{\beta \frac{| x-y|^2}{s}} dy  
\le C s^{-d/2 -1} e^{6 \beta t^2/s} \| u \|_1^2. 
\end{equation}
for some constants $C$ and $\beta > 0$ which are  independent of $n$.  
Let $k \in \{ 1,\ldots,d \} $.
Then
$$| (e^{-sH_n} M_\chi u, \partial_k M_\chi (\varphi \, e^{\beta \frac{| x- . |^2}{2s}}) ) | 
\le C s^{-\frac{d}{4} - \frac{1}{2}} e^{3 \beta t^2/s} \| u \|_1 \| \varphi \|_2$$
for all $\varphi \in C_c^\infty(\Ri^d)$.
On the other hand,  $e^{-tH_n}$ converges strongly in $L^2(\RR^d)$  to $e^{-tH}$
(see \cite{AE2} Corollary~3.9).
It follows then that 
$$| (e^{-sH} M_\chi u, \partial_k M_\chi (\varphi \, e^{\beta \frac{| x- . |^2}{2s}}) ) |
 \le C s^{-\frac{d}{4} - \frac{1}{2}} e^{3 \beta t^2/s} \| u \|_1 \| \varphi \|_2.$$ 
Since this is true for  all $\varphi \in C_c^\infty(\RR^d)$ we have by density 
$$ \| (M_\chi \partial_k  e^{-sH} M_\chi u) \cdot e^{\beta \frac{| x-.|^2}{2s}}  \|_2 
\le C s^{-\frac{d}{4} - \frac{1}{2}} e^{3 \beta t^2/s} \| u \|_1.$$
This proves the lemma.
\end{proof} 

\begin{proof}[Proof of Theorem \ref{th7}] 
It follows from Lemma~\ref{lemm1} that the Riesz transform 
$M_\chi \partial_k A^{-1/2}$ is bounded on $L^2(\RR^d)$. 

In order to prove a weak type estimate for $T = M_\chi  \partial_k H^{-1/2} M_\chi$ 
we apply Theorem \ref{th1} with $S = M_\chi \partial_k H^{-1/2}$ and $A_t = e^{-t^2 H} M_\chi$. 
These operators are bounded on $L^2(\RR^d)$ and by Lemma \ref{le2}
the operators $A_t$ satisfy assumption \eqref{au2}. 
It remains then to check (\ref{DMop}).
 By the formula 
$$H^{-1/2} = \frac{1}{2 \sqrt{\pi}} \int_0^\infty e^{-s H} \frac{ds}{\sqrt{s}}$$
we have 
$$ H^{-1/2}e^{-t^2H} = \frac{1}{2 \sqrt{\pi}} \int_0^\infty e^{-(s+t^2)H} \frac{ds}{\sqrt{s}}= \frac{1}{2 \sqrt{\pi}} \int_0^\infty e^{-sH} \one_{\{s> t^2\}} \frac{ds}{\sqrt{s-t^2}}.$$
Let $\beta > 0$ be as in Lemma~\ref{lemm3} and let $\delta > 0$.
Fix $x \in \RR^d$, $t > 0$ and let $u \in L^2(\RR^d)$ with $\supp u \subset B(x,t)$. Set $\nu(s,t) = \Big| \one_{\{s> t^2\}} \frac{1}{\sqrt{s-t^2}} - \frac{1}{\sqrt{s}} \Big|$.
Then
\begin{eqnarray*}
\lefteqn{
2 \sqrt{\pi} \int_{\RR^d \setminus B(x,(1+ \delta)t)} | ((T - SA_t)u)(y) | \, dy
} \hspace{5mm} \\*
 &\le&    \int_0^\infty \int_{\RR^d \setminus B(x,(1+ \delta)t)} 
 | (M_\chi \partial_k e^{-s H} M_\chi u) (y) | \,  dy \, 
\nu(s,t) \, ds\\
 &\le&  \int_0^\infty\nu(s,t)   \left( \int_{\RR^d} | (M_\chi \partial_k e^{-s H} M_\chi u) (y) |^2 \, 
    e^{\beta \frac{| x-y|^2}{s}} dy \right)^{1/2} \times \\*
 && \hspace{50mm} {} \times  \Big( \int_{\RR^d \setminus B(x,(1+ \delta)t)} 
      e^{-\beta \frac{| x-y|^2}{s}} \, dy\Big)^{1/2} ds\\
 &\le& C \int_0^\infty \nu(s,t) \,
       s^{-\frac{d}{4} - \frac{1}{2}} e^{3 \beta t^2/s} \|u\|_1 
 \Big( \int_{\RR^d \setminus B(x,(1+ \delta)t)} e^{-\beta \frac{| x-y|^2}{s}}  dy \Big)^{1/2} ds.
 \end{eqnarray*}
Note that we have used Lemma \ref{lemm3}  in the last inequality.
Now 
$$  \int_{\RR^d \setminus B(x,(1+ \delta)t)} e^{-\beta \frac{| x-y|^2}{s}} \, dy
\le  e^{-\beta \frac{(1+\delta)^2 t^2}{2s}} 
    \int_{\RR^d } e^{-\beta \frac{| x-y|^2}{2s}} \, dy
\le C s^{d/2} e^{-\beta \frac{(1+\delta)^2 t^2}{2s}}.$$
Choosing $\delta \ge 4$ we obtain a positive constant $\gamma$ such that 
\[
 \int_{\RR^d \setminus B(x,(1+ \delta)t)} | (T - SA_t)u(y) | \, dy 
\leq C \int_0^\infty \nu(s,t)
   s^{-1/2} e^{-\gamma t^2/s} \, ds.
\]
The latter integral is bounded by some constant  $M$  independent of $t$.
 This proves (\ref{DMop}) and hence $T = M_\chi  \partial_k H^{-1/2} M_\chi$ is weak type
$(1,1)$.
By interpolation, it is bounded on $L^p(\RR^d)$ for all $1 < p \le 2$.
\end{proof}

As discussed at the end of the previous section, we note that if one proves a 
version of  Theorem \ref{thEO} without the extra factor $(1+t)^{d/2}$ if
 $(a_{kj}(x)) \ge \mu I$ for a.e.\ $x$ in a  `big' domain, then
Theorem \ref{th7} holds with $A$ in place of $H$.
That is $M_\chi  \partial_k A^{-1/2} M_\chi$ is weak type $(1,1)$ and bounded on $L^p(\RR^d)$ 
for all $1 < p \le 2$.
In \cite{EO2}, we prove by a different method that 
if the coefficients $a_{kj} \in W^{1,\infty}(\Ri^d)$, 
then $M_\chi \partial_k (I+A)^{-1/2} $ and $M_\chi \partial_k \, \partial_j (I + A)^{-1}$ 
are bounded on $L^p(\RR^d)$ for all $p \in (1,\infty)$.
Moreover, if $a_{kj} \in W^{\nu,\infty}(\Ri^d,\CC)$ then we show that 
$M_\chi \, \partial_k (I + A)^{-1/2} \, M_\chi$ is bounded on $L^p$
for all $p \in (1,\infty)$.

\vspace{.3cm}
\noindent {\bf Acknowledgements} Most of this work was carried out whilst  the second named
 author was visiting the  Department of Mathematics at the University of  Auckland. 
He wishes to thank both the University of Auckland and  the CNRS for financial support. 
 Part of this work is supported by the Marsden Fund Council from Government funding, 
administered by the Royal Society of New Zealand.

\newcommand{\etalchar}[1]{$^{#1}$}

\end{document}